\documentclass[english,12pt]{smfart}

\usepackage[utf8]{inputenc}
\usepackage[T1]{fontenc}
\usepackage[english]{babel}

\usepackage{amssymb,url,xspace,smfthm}

\newcommand{\BibTeX}{{\scshape Bib}\kern-.08em\TeX}
\newcommand{\T}{\S\kern .15em\relax }
\newcommand{\AMS}{$\mathcal{A}$\kern-.1667em\lower.5ex\hbox
        {$\mathcal{M}$}\kern-.125em$\mathcal{S}$}

\newcommand{\field}[1]{\mathbb{#1}}

\newcommand{\RR}{\field{R}}

\def\vect#1{\overrightarrow{#1}}

\def\d{\mathrm{d}}

\def\Reflx{\mathop{\rm Reflec}\nolimits}
\def\Refrac{\mathop{\rm Refrac}\nolimits}
\title[A direct proof of Malus' theorem]{A direct proof of Malus' theorem using the symplectic sructure of the set of oriented straight lines}
\date {\today}
\author[ C.-M.~Marle]{Charles-Michel Marle}

\address{27 avenue du 11 novembre 1918 \\92190 Meudon\\ France}
\email{charles-michel.marle@math.cnrs.fr}
\urladdr{http://marle.perso.math.cnrs.fr/}


\dedicatory{This modest work is dedicated to the memory\\ of the former students of the French 
\'Ecole Polytechnique\\ \'Etienne Louis Malus (X 1794) and Pierre Charles François Dupin (X 1801),\\ in the hope that 
this school will in the future, as it did in the past, produce\\ high level scientists and engineers rather than bankers and traders.}

\begin{document}
\def\smfbyname{}

\begin{abstract}
We present a direct proof of Malus' Theorem in Geometrical Optics founded on the symplectic structure of the
set of all oriented straight lines in an Euclidean affine space.
\end{abstract}

\begin{altabstract}
Nous présentons une preuve directe du théorème de Malus de l'optique géométrique basée sur la structure symplectique
de l'ensemble des droites orientées d'un espace affine euclidien.
\end{altabstract}
\maketitle


\section{Introduction}

Geometric Optics is a physical theory in which the propagation of light is described in terms of \emph{light rays}.
In this theory, the physical space in which we live and in which the light propagates is treated, once a unit of lenghth
is chosen, as a three-dimensional Euclidean affine space $\mathcal E$. In a transparent homogeneous medium occupying
the whole physical space $\mathcal E$, a light ray is described by an oriented straight line drawn in that space. 
When the transparent medium occupies only a part of $\mathcal E$, a light ray, as long as it is contained in that medium,
is described by an oriented segment of a straight line, but it will be convenient to consider the full oriented straight line  
which supports that segment. Reflections on smooth reflecting surfaces, or refractions through smooth 
surfaces separating two transparent media with different refractive indices, which transform an incident light 
ray into the corresponding reflected or refracted light ray, therefore appear as  transformations, defined on a part 
of the set $\mathcal L$ of all oriented straight lines in $\mathcal E$, with values in $\mathcal L$.
\par\smallskip

The set $\mathcal L$ of all possible oriented straight lines drawn in the three-dimensional Euclidean affine
space $\mathcal E$ depends on four parameters: an oriented line $L\in{\mathcal L}$ being given, two parameters 
are needed to specify the unit vector $\vect u$ parallel to and of the same orientation as $L$, and two more parameters, 
for example the coordinates of its intersection point with a transverse fixed plane,
are needed to specify the position of $L$. We will prove below that $\mathcal L$ has the structure of a smooth
four-dimensional symplectic manifold\footnote{More generally, the space of oriented straight lines in an $n$-dimensional
Euclidean affine space is a $2(n-1)$-dimensional symplectic manifold.}.
\par\smallskip

Very often in Geometrical Optics one deals with the propagation of light rays which  make a sub-family 
of the family of all possible light rays, smoothly depending on a number of parameters smaller than 4. 
For example, the family of light rays emitted by a luminous point in all possible directions, or the family of
light rays emitted by a smooth luminous surface, each ray being emitted in the direction orthogonal to the surface, smoothly
depend on $2$ parameters.  Let us state two definitions.

\begin{defi}\label{DefinitionRank}
The \emph{rank} of a family of rays which smoothly depend on a finite number of parameters is the number of 
these parameters.
\end{defi}
 
\begin{defi}\label{NormalFamily} 
A rank 2 family of rays is said to be \emph{normal} if 
at each point of the lighted part of the physical space $\mathcal E$ there exists 
a small piece of smooth surface
orhtogonal to all the rays which meet it. 
\end{defi}

In \cite{Hamilton1827}, Hamilton uses a slightly different terminology: he calls \emph{class} the rank of a system
of rays and \emph{orthogonal system} a normal system of rays.
\par\smallskip

In an homogeneous and isotropic medium, the family of light rays
emitted by a luminous point is a normal family: the spheres centered on the luminous point are indeed orthogonal to all rays.
The family of rays emitted by a smooth luminous surface which, at each of its point,
emits a ray in the direction orthogonal to the surface, too is a normal family: it is indeed a well known 
geometric property of the straight lines orthogonal to a smooth surface.
\par\smallskip

Malus' theorem states that a rank 2 normal family of rays which propagates through an optical system made of 
any number of homogeneous and isotropic transparent media, with any number of smooth reflecting or refracting 
surfaces, always remains normal, in each of the media through which it propagates.
\par\smallskip

After a short presentation of its historical backgroud in Section \ref{history}, we give a direct proof of Malus' theorem
which rests on symplectic geometry. We show in Section \ref{SymplecticStructure}
that the set $\mathcal L$ of all oriented straight lines
in the Euclidean affine space $\mathcal E$ is a smooth $4$-dimensional manifold, naturally endowed with a symplectic
$2$-form $\omega$.This symplectic form  is the pull-back, by a diffeomorphism of
$\mathcal L$ onto the cotangent bundle $T^*\Sigma$ to a two-dimensional sphere $\Sigma$, of the canonical symplectic form
on that cotangent bundle. That diffeomorphism of $\mathcal L$ onto $T^*\Sigma$ is not uniquely determined:
it depends on the choice of a reference point in $\mathcal E$; but the symplectic form $\omega$
does not depend on that choice. In Section \ref{ReflectionRefraction} we prove that reflection on a smooth surface
is a symplectic diffeomorphism of one open subset of the symplectic manifold $({\mathcal L},\omega)$
(made by light rays which hit the mirror on its reflecting side) onto another open subset of that manifold.
Similarly, we prove that refraction through a smooth surface which separates two transparent media of refractive indices
$n_1$ and $n_2$ is a symplectic diffeomorphism of one subset of $\mathcal L$ endowed with the symplectic form
$n_1\omega$ onto another open subset of that space endowed with the symplectic form $n_2\omega$.
In Section \ref{NormalLagrangian} we prove that a rank 2 family of rays is normal if and only if it is a Lagrangian
submanifold of $({\mathcal L},\omega)$. We conclude in Section \ref{Conclusion}: since a symplectic 
diffeomorphism maps Lagrangian submanifolds onto Lagrangian submanifolds, Malus' theorem immediately follows 
from the results obtained in previous Sections.
\par\smallskip

For much more elaborate applications of symplectic geometry in Optics, the readers are referred to
\cite{Arnold} (Chapter 9, Section 46, pp. 248--258 and Appendix 11, pp. 438--439) and \cite{GuilleminSternberg1984} 
(Introduction, pp. 1--150). 

\section{Historical background}\label{history}

\'Etienne Louis Malus (1775--1812) 
is a French scientist who investigated geometric properties of families of straight oriented lines, in view of applications to light rays. 
He developed Huygens undulatory theory of light, discovered and investigated the phenomenon of 
\emph{ polarization of light} and the phenomenon of \emph{ double refraction} of light in crystals. 
He participated in Napoleon's disastrous expedition to Egypt (1798--1801)  where he contracted diseases probably
responsible for his early death.
He proved \cite{Malus1808} that the family of rays emitted by a luminous point source (which, as
we have seen above, is normal) remains normal after \emph{ one reflection} on a smooth surface, or 
\emph{ one refraction} through a smooth surface but 
he was not sure \cite{Malus1811} whether this property remains true for \emph{ several} 
reflections or refractions~\footnote{This is an illustration of the famous Arnold theorem: \emph{when a theorem
or a mathematical concept bears the name of a scientist, it is not due to that scientist}. 
Arnold added: of course, that theorem applies to itself!}.
His works on families of oriented straight lines were later used 
and enhanced by Hamilton. 
\par\smallskip

For reflections, a very simple geometric proof of Malus' theorem was obtained by
the French scientist Charles Dupin \cite{Dupin1816}. For this reason, 
in French Optics manuals \cite{Courty}, Malus' theorem is frequently called Malus-Dupin's theorem. 
According to \cite{ConwaySynge}, Quetelet and Gergonne gave a full proof of Malus' theorem 
for refractions in 1825. Independently, the great Irish mathematician William Rowan Hamilton
(1805--1865) gave a proof of this theorem, both for reflections and for refractions, 
in his famous paper \cite{Hamilton1827}. His 
proof rests on the stationarity properties of the optical length of rays, with respect to 
infinitesimal displacements of the points of reflections or of refractions, 
on the reflecting or refracting surfaces. 
\par\smallskip

Charles François Dupin (1784--1873) is a French
mathematician and naval engineer. Several mathematical objects in differential geometry
bear his name: \emph{Dupin cyclids} (remarkable surfaces he discovered when he was
a young student at the French \'Ecole Polytechnique), \emph{Dupin indicatrix} (which describes the local shape
of a surface). It is said in Wikipedia \cite{Dupin} 
that he inspired to the poet and novelist Edgar Allan Poe 
(1809--1849) the figure of \emph{ Auguste Dupin} appearing in the three 
detective stories: \emph{The murders in the rue Morgue}, 
\emph{The Mystery of Marie Roget} and \emph{The Purloined Letter}.

\section{The symplectic structure of the set of oriented straight lines}\label{SymplecticStructure}

\begin{prop}
The set $\mathcal L$ of all possible oriented straight lines in 
the affine $3$-dimensional Euclidean space $\mathcal E$ has a natural structure of smooth
$4$-dimensional manifold, is endowed with a symplectic form $\omega$ and is diffeomorphic,
by a symplectic diffeomorphism, to the cotangent bundle $T^*\Sigma$ to a 2-dimensional sphere $\Sigma$ .
\end{prop}

\begin{proof}
Let indeed $\Sigma$ be a sphere of any fixed radius $R$ (for example $R=1$) centered on a point $C\in{\mathcal E}$  
and $O$ be another fixed point in $\mathcal E$. Of course we can take $O=C$, but for clarity it is better to separate 
these two points. An oriented straight line $L$ determines
\begin{itemize}

\item{} a unique point $m\in\Sigma$ such that the vector $\vect u=\vect{Cm}$ 
is parallel to and of the same direction as $L$,

\item{} a unique linear form $\eta$ on the tangent space $T_m\Sigma$
at $m$ to the sphere $\Sigma$, given by
 $$\eta(\vect w)=\vect{OP}\cdot\vect w\quad \hbox{for all}\quad 
 \vect w\in T_m\Sigma\,,
 $$
where $P$ is any point of the line $L$, and where $\vect{OP}\cdot\vect w$ denotes the scalar product of the vectors
$\vect{OP}$ and $\vect w$.
\end{itemize}

The pair $(m,\eta)$ is an element of the \emph{ cotangent bundle} 
$T^*\Sigma$. In fact $m$ being determined by $\eta$, we can say that $\eta$ is an element of $T^*\Sigma$.
\par\smallskip

Conversely, an element $\eta\in T^*\Sigma$, \emph{i.e.} a linear form $\eta$ on the tangent space to $\Sigma$ at some point $m\in\Sigma$, determines an oriented
straight line $L$, parallel to and of the same direction as $\vect u=\vect{Cm}$. 
This line is the set of points $P\in{\mathcal E}$ such that
 $$\vect{OP}\cdot \vect w=\eta(\vect w)\quad\hbox{for all } \vect w\in T_m\Sigma\,.
 $$ 
There exists (\cite{LibermannMarle} p. 59, \cite{Camille} p. 176) on the cotangent bundle $T^*\Sigma$ a unique
differential one-form $\lambda_\Sigma$, called the \emph{
Liouville form}, whose exterior differential $\d \lambda_\Sigma$ is a 
symplectic form on $T^*\Sigma$. The above described
1--1 correspondence between the set $\mathcal L$ of all oriented straight lines 
and the cotangent bundle $T^*\Sigma$ allows us to transport on $\mathcal L$ 
the differentiable manifold structure, the Liouville one form $\lambda_\Sigma$ and the symplectic 
form $\d \lambda_\Sigma$. So we get on $\mathcal L$ a differential one-form $\lambda_O$ 
and a symplectic form $\omega=\d\lambda_O$. Therefore $({\mathcal L},\omega)$
is a \emph{ symplectic manifold}. 
\par\smallskip

The diffeomorphism so obtained, the one-form $\lambda_O$ and its exterior
differential $\omega=\d\lambda_O$ \emph{ do not depend on} the choice of
the centre $C$ of the sphere $\Sigma$ (with the obvious convention that two spheres 
of the same radius centered on two different points $C$ and $C'$ are identified by means 
of the translation which sends $C$ on $C'$). 
However, this diffeomorphism
\emph{ depends on} the choice of the point $O$, and so does the one-form
$\lambda_O$: when, to a given straight line $L$, the choice of $O$ associates 
the pair $(m,\eta)\in T^*\Sigma$, the choice of another point $O'$ 
associates the pair $\bigl(m,\eta+\d f_{O'O}(m)\bigr)$, where 
$f_{O'O}:\Sigma\to\RR$ is the smooth function defined on $\Sigma$
 $$f_{O'O}(n)=\vect{O'O}\cdot\vect{Cn}\,,\quad n\in\Sigma\,.
 $$
Therefore, if the choice of $O$ determines on $\mathcal L$ the one-form $\lambda_O$, the choice of $O'$ determines the one-form
 $$\lambda_{O'}=\lambda_O+\d (f_{O'O}\circ \pi_\Sigma)
 $$
where $\pi_\Sigma:T^*\Sigma\to\Sigma$ is the canonical projection.
The symplectic form  $\omega$ on the set of all oriented
straight lines $\mathcal L$ \emph{ does not depend on the choice
of $O$, nor on the choice of $C$}, since we have 
 $$\omega=\d\lambda_O=\d\lambda_{O'}\quad \hbox{because}\ \d\circ\d=0\,.\qedhere
 $$  
\end{proof}

\begin{prop}
Let $(\vect e_1,\vect e_2,\vect e_3)$ be an orthonormal basis of the Euclidean vector space $\vect{\mathcal E}$ associated to the affine Euclidean space $\mathcal E$. Any oriented straight line $L\in{\mathcal L}$ can be determined by its unit directing vector $\vect u$ and by a point $P\in L$ (determined up to addition of a vector collinear with $\vect u$).  Expressed in terms of the coordinates $(p_1,p_2,p_3)$ of $P$ in the affine frame 
$(O,\vect e_1,\vect e_2,\vect e_3)$ and of the components $(u_1,u_2,u_3)$ of 
$\vect u$, the symplectic form $\omega$ is given by
 $$\omega=\sum_{i=1}^3\d p_i\wedge\d u_i\,.
 $$
\end{prop}

\begin{proof}
Using the definition of the Liouville one-form on 
$T^*\Sigma$, we see that
 $$\lambda_0=\sum_{i=1}^3 p_i\d u_i\,.
 $$
Therefore
 $$\omega=\d\lambda_O=\sum_{i=1}^3\d p_i\wedge\d u_i\,.\qedhere
 $$
\end{proof}

\begin{rema} The three components $u_1$, $u_2$, $u_3$ of $\vect u$ 
are not independent, 
since they must satisfy $\sum_{i=1}^3(u_i)^2=1$. The point $P\in L$ used to detemine the oriented straight line $L$ is not uniquely determined, since by adding to $P$ any vector collinear with $\vect u$ we get another point in $L$. However, these facts do not affect the validity of the expression of $\omega$
given above.
\end{rema}

\begin{rema}
The symplectic form $\omega$ can be expressed very concisely by using an obvious vector notation combining the wedge and scalar products:
 $$\omega(P,\vect u)=\d \vect P\wedge\d \vect u\,.
 $$
\end{rema}

\section{Reflection and  refraction are symplectic diffeomorphisms}\label{ReflectionRefraction}

\begin{prop} Let $M$ be a smooth reflecting surface. Let $\Reflx_M$ be the map which associates to each light ray $L_1$ which hits $M$ on its reflecting side, the reflected light ray $L_2=\Reflx_M(L_1)$. The map 
$\Reflx_M$ is a  \emph{ symplectic diffeomorphism} defined on
the open subset of the symplectic manifold $({\mathcal L},\omega)$ made by light rays which hit $M$on its reflecting side, onto the open subset made by the same
straight lines  with the opposite orientation.
\end{prop}

\begin{proof}
Any oriented straight line $L_1$ which hits the mirror $M$ is determined by

\begin{itemize}
\item{} the unit vector ${\vect u}_1$ parallel to and of the same direction as $L_1$,

\item{} the incidence point $P\in M$ of the light ray on the mirror.
\end{itemize}

We will write $\vect P$ for the vector $\vect{OP}$, the fixed point $O$ being arbitrarily chosen.

The reflected ray $L_2$ is determined by

\begin{itemize}

\item{} the unit vector ${\vect u}_2$, given in terms of ${\vect u}_1$ 
by the formula
 $${\vect u}_2={\vect u}_1+2({\vect u}_1\cdot\vect n)\vect n\,,
 $$
where $\vect n$ is a unit vector orthogonal to the mirror $M$ at the incidence point
$P$, with anyone of the two possible orientations;

\item{} the same point $P\in M$ on the mirror.

\end{itemize}

According to the expression of the symplectic form $\omega$ given in the last
Remark, we have to check that
$\d \vect P\wedge\d {\vect u}_2=\d\vect P\wedge \d{\vect u}_1
$. 
We have

 \begin{align*}
 \d\vect P\wedge\d({\vect u}_2-{\vect u}_1)
 &=2\d\vect P\wedge\d\bigl((\vect u_1\cdot\vect n)\vect n\bigr)\\
 &=-2\d\bigl((\vect u_1\cdot\vect n)(\vect n\cdot\d\vect P)\bigr)\\
 &=0\,,
 \end{align*}
because $\vect n\cdot \d\vect P=0$, the differential $\d\vect P$ lying tangent to the mirror $M$, while the vector $\vect n$ is 
orthogonal to the mirror.
\end{proof}

\begin{prop} Let $R$ be a smooth refracting surface, which separates two
transparent media with refractive indexes $n_1$ and $n_2$. Let $\Refrac_R$ be the map which associates, to each light ray $L_1$ which hits the refracting surface $R$
on the side of refracting index $n_1$, the corresponding refracted ray 
$L_2=\Refrac_R(L_1)$ determined by Snell's law of refraction. The map $\Refrac_R$
is a  \emph{ symplectic diffeomorphism} defined on an 
\emph{ open subset} of $(\mathcal L,n_1\omega)$, (the set of oriented straight lines which hit $R$ on the $n_1$ side and, if $n_1>n_2$, are not totally reflected)
with values in an \emph{open subset} of $(\mathcal L,n_2\omega)$.
\end{prop}

\begin{proof}
Any oriented straight line $L_1$ which hits the refracting surface $R$ is determined by

\begin{itemize}
\item{} the unit vector ${\vect u}_1$ parallel to and of the same direction as $L_1$,

\item{} the incidence point $P\in R$ of the light ray on the refracting surface.
\end{itemize}

We will write $\vect P$ for the vector $\vect{OP}$, the fixed point $O$ being arbitrarily chosen.

The refracted ray $L_2$ is determined by

\begin{itemize}

\item{} the unit vector ${\vect u}_2$, related to ${\vect u}_1$ by
the formula
 $$n_2\bigl({\vect u}_2-(\vect u_2\cdot\vect n)\vect n\bigr)
=  n_1\bigl({\vect u}_1-(\vect u_1\cdot\vect n)\vect n\bigr)\,,
 $$
where $\vect n$ is a unit vector orthogonal to the refractig surface $R$ at 
the incidence point $P$, with anyone of the two possible orientations;

\item{} the same point $P\in R$ on the refracting surface.

\end{itemize}

We have to check that $n_2\d \vect P\wedge\d {\vect u}_2=n_1\d\vect P\wedge 
\d{\vect u}_1$. 
We have

 \begin{align*}
 \d\vect P\wedge(n_2\d{\vect u}_2-n_1\d{\vect u}_1)
 &=\d\vect P\wedge\d\bigl(n_2(\vect u_2\cdot\vect n)\vect n
                         -n_1(\vect u_1\cdot\vect n)\vect n\bigr)\\
 &=-\d\Bigl(\bigl(n_2(\vect u_2\cdot\vect n)
                 -n_1(\vect u_1\cdot\vect n)\bigr)(\vect n\cdot\d\vect P)\Bigr)\\
 &=0\,,
 \end{align*}
because $\vect n\cdot\d\vect P=0$, the differential $\d\vect P$ lying tangent to the 
refracting surface $R$, while the vector $\vect n$ is orthogonal to $R$.
\end{proof}

\section{Normal families are Lagrangian submanifolds}\label{NormalLagrangian}

\begin{prop} A rank 2 family of oriented straight lines
is  normal (in the sense of Definition \ref{NormalFamily}) if and only if it is a Lagrangian 
submanifold (\cite{LibermannMarle} p. 92, or \cite{OrtegaRatiu} p. 123) of the symplectic manifold $({\mathcal L},\omega)$ of 
all oriented straight lines.
\end{prop}

\begin{proof}
Let us consider a rank 2 family of oriented straight lines. Locally, in a neighbourhood of each of its straight lines, the family can be determined by a smooth map $(k_1,k_2)\mapsto L(k_1,k_2)$, defined on an open substet of $\RR^2$,
with values in the manifold $\mathcal L$ of oriented straight lines.
For each value of $(k_1,k_2)$, the ray $L(k_1,k_2)$ can be determined by
\begin{itemize}
\item{} a point $P(k_1,k_2)$ of the ray $L(k_1,k_2)$,

\item{} the unit director vector $\vect u(k_1,k_2)$ of the ray $L(k_1,k_2)$
\end{itemize}

Although $P(k_1,k_2)$ is not uniquely determined, we can arrange things so that the map $(k_1,k_2)\mapsto\bigl(P(k_1,k_2),\vect u(k_1,k_2)\bigr)$ is smooth. By assumption it is everywhere of rank 2.

The reciprocal image of the symplectic form $\omega$ of $\mathcal L$ by the map
$(k_1,k_2)\mapsto \bigl(P(k_1,k_2),\vect u(k_1,k_2)\bigr)$ is
 $$\left(\frac{\partial\vect P(k_1,k_2)}{\partial k_1}\cdot
         \frac{\partial\vect u(k_1,k_2)}{\partial k_2} - 
         \frac{\partial\vect P(k_1,k_2)}{\partial k_2}\cdot
         \frac{\partial\vect u(k_1,k_2)}{\partial k_1}\right)\d k_1\wedge\d k_2\,,
 $$
where, as before, we have written $\vect P(k_1,k_2)$ for $\vect{OP}(k_1,k_2)$, the origin $O$ being any fixed point in $\mathcal E$.  

Using the symmetry property of the second derivatives
 $$\frac{\partial^2\vect P(k_1,k_2)}{\partial k_1\partial k_2}=
   \frac{\partial^2\vect P(k_1,k_2)}{\partial k_2\partial k_1}
 $$
we see that the reciprocal image of $\omega$ can be written
 $$\left(\frac{\partial}{\partial k_2}\left(\vect u\cdot
                                       \frac{\partial\vect P}
                                            {\partial k_1}\right) -
         \frac{\partial}{\partial k_1}\left(\vect u\cdot
                                       \frac{\partial\vect P}
                                            {\partial k_2}\right) \right)
 \d k_1\wedge\d k_2\,.
 $$
where we have written $\vect u$ and $\vect P$ for $\vect u(k_1,k_2)$ and 
$\vect P(k_1,k_2)$.

Our rank 2 family of rays is a Lagrangian submanifold of $({\mathcal L},\omega)$
if and only if the reciprocal image of $\omega$ vanishes, \emph{i.e.}, if and only if
 $$\frac{\partial}{\partial k_2}\left(\vect u\cdot
                                       \frac{\partial\vect P}
                                            {\partial k_1}\right)=
         \frac{\partial}{\partial k_1}\left(\vect u\cdot
                                       \frac{\partial\vect P}
                                            {\partial k_2}\right)\,,
 $$
or if and only if there exists locally a smooth function 
$(k_1,k_2)\mapsto F(k_1,k_2)$ such that
 $$\vect u\cdot\frac{\partial\vect P}{\partial k_1}=
    \frac{\partial F}{\partial k_1}\,,\quad
   \vect u\cdot\frac{\partial\vect P}{\partial k_2}=
    \frac{\partial F}{\partial k_2}\,.\eqno(*)
 $$

Let us now look at the necessary and sufficient conditions under which there exists
locally, near a given ray of the family, a smooth surface orthogonal to all the neighbouring rays of the family. This surface is the image of a map
 $$(k_1,k_2)\mapsto P(k_1,k_2)+\lambda(k_1,k_2)\vect u(k_1,k_2)\,,
 $$
where $(k_1,k_2)\mapsto\lambda(k_1,k_2)$ is a smooth function.

This surface is orhtogonal to the rays if and only if the function $\lambda$ is such that
 $$ \vect u(k_1,k_2)\cdot\d\Bigl(\vect P(k_1,k_2)+\lambda(k_1,k_2)
                              \vect u(k_1,k_2)\Bigr)=0\,.
 $$
The equalities $\vect u(k_1,k_2)\cdot\d\vect u(k_1,k_2)=0$ and
$\vect u(k_1,k_2)\cdot\vect u(k_1,k_2)=1$ allow us to write this condition as
$$ \vect u(k_1,k_2)\cdot\d \vect P(k_1,k_2)+\d\lambda(k_1,k_2)
                              =0\,.\eqno(**)
$$
We see that when there exists a smooth
function $F$ which satifies $(*)$, all functions $\lambda=-F+$ Constant satisfy 
$(**)$, and conversely when there exists a smooth function $\lambda$ which satisfies $(**)$, all functions $F=-\lambda$ + Constant satisfy $(*)$. A rank 2 family of rays is therefore normal il and only if it is a Lagrangian submanifold of $\mathcal L$.
\end{proof}

\section{Malus' Theorem}\label{Conclusion}

Since reflections and refractions are 
symplectic diffeomorphisms, and since by composing several symplectic diffeomorphisms we get again 
a symplectic diffeomorphism, the travel of light rays through an optical device with any number of reflecting 
or refracting smooth surfaces is a symplectic diffeomorphism.
The image of a Lagrangian submanifold by a symplectic diffeomorphism is automatically a Lagrangian submanifold.
We therefore can state as a theorem the following result, very remarkable by the fact that
no assumption other than their smoothness is made about the shapes of the reflecting or refracting surfacesse surfaces. 
It is only assumed that these surfaces are smooth and that the reflections or refractions obey the well-known laws of Optics. 

\begin{theo}[Malus' Theorem] A two parameter normal family of light rays remains normal after any number of reflections
on smooth reflecting surfaces or refractions across smooth 
surfaces which separate transparent media with different refractive indexes.  

\end{theo}

\section{Acknowledgements}

I address my warmest thanks to Géry de Saxcé for offering me to present a talk at the 58-th Souriau Colloquium, and for all
the efforts he made and is still making for the success and the perennity of these colloquia.
\par\smallskip

I am indebted to Dominique Flament, who recently gave me the opportunity to present at his seminar 
\emph{Histoires de Géométries} the works of William Rowan Hamilton on Geometrical Optics. 
It was in preparing my talk for this seminar that I learnt about Malus' theorem and tried to find a 
direct proof founded in symplectic geometry.

\backmatter

\end{document}